\documentclass[fleqn,reqno,11pt,a4paper,final]{amsart}

\usepackage[a4paper,left=30mm,right=30mm,top=30mm,bottom=29mm,marginpar=20mm]{geometry}
\usepackage{mathtools}     
\mathtoolsset{showonlyrefs}
\usepackage{xcolor} 
\usepackage{amsmath}
\usepackage{amssymb}
\usepackage{amsthm}
\usepackage{amscd}
\usepackage[ansinew]{inputenc}
\usepackage{cite}
\usepackage{bbm}
\usepackage{color}
\usepackage[english=american]{csquotes}
\usepackage[final]{graphicx}
\usepackage[colorlinks=true, pdfstartview=FitV, linkcolor=black, citecolor=black, urlcolor=black]{hyperref}
\usepackage{calc}
\usepackage{mathptmx}
\usepackage{t1enc}
\usepackage{stackrel}

\DeclareMathAlphabet{\mathcal}{OMS}{cmsy}{m}{n}

\numberwithin{equation}{section}
\newtheoremstyle{thmlemcorr}{10pt}{10pt}{\itshape}{}{\bfseries}{.}{10pt}{{\thmname{#1}\thmnumber{ #2}\thmnote{ (#3)}}}
\newtheoremstyle{thmlemcorr*}{10pt}{10pt}{\itshape}{}{\bfseries}{.}\newline{{\thmname{#1}\thmnumber{ #2}\thmnote{ (#3)}}}
\newtheoremstyle{remexample}{10pt}{10pt}{}{}{\bfseries}{.}{10pt}{{\thmname{#1}\thmnumber{ #2}\thmnote{ (#3)}}}
\newtheoremstyle{ass}{10pt}{10pt}{}{}{\bfseries}{.}{10pt}{{\thmname{#1}\thmnumber{ A#2}\thmnote{ (#3)}}}

\theoremstyle{thmlemcorr}
\newtheorem{theorem}{Theorem}
\numberwithin{theorem}{section}
\newtheorem{lemma}[theorem]{Lemma}

\newtheorem{problem}[theorem]{Problem}

\newtheorem{definition}[theorem]{Definition}

\theoremstyle{remexample}
\newtheorem{remark}[theorem]{Remark}
\newtheorem{example}[theorem]{Example}

\theoremstyle{ass}

\newcommand{\Rbb}{\mathbb{R}}
\newcommand{\Tbb}{\mathbb{T}}
\newcommand{\T}{\mathbb{T}}

\DeclareMathOperator{\dist}{dist}

\DeclareMathOperator{\id}{id}
\DeclareMathOperator{\diverg}{div}

\DeclareMathOperator{\curl}{curl}
\DeclareMathOperator{\supp}{supp}

\newcommand{\dd}{\;\mathrm{d}}

\newcommand{\N}{\mathbb{N}}
\newcommand{\R}{\mathbb{R}}

\newcommand{\Z}{\mathbb{Z}}

\newcommand{\eps}{\epsilon}
\def\d{{\,\rm d}}

\newcommand{\image}{\mathrm{im}\;}

\renewcommand{\eps}{\varepsilon}
\renewcommand{\epsilon}{\varepsilon}
\renewcommand{\phi}{\varphi}

\def\XXint#1#2#3{{\setbox0=\hbox{$#1{#2#3}{\int}$ }
		\vcenter{\hbox{$#2#3$ }}\kern-.58\wd0}}

\begin{document}


\title[Lower Semi-Continuity under Convex Constraints]{Lower Semi-Continuity for $\mathcal A$-quasiconvex 
Functionals under Convex Restrictions}


%

\author{Jack Skipper}
\address{\textit{Jack Skipper:} Institute of Applied Analysis, Ulm University, Helmholtzstrasse~18, 89081 Ulm, Germany}
\email{jwd.skipper@gmail.com}

\author{Emil Wiedemann}
\address{\textit{Emil Wiedemann:} Institute of Applied Analysis, Ulm University, Helmholtzstrasse~18, 89081 Ulm, Germany}
\email{emil.wiedemann@uni-ulm.de}


\begin{abstract}

	We show weak lower semi-continuity of functionals assuming the new notion of a ``convexly constrained''  $\mathcal A$-quasiconvex integrand. 
	We assume $\mathcal A$-quasiconvexity
	only for functions defined on a set $K$ which is convex. Assuming this and sufficient integrability of the sequence we show that the functional is still (sequentially) weakly lower semi-continuous along weakly convergent ``convexly constrained''  $\mathcal A$-free sequences. In a motivating example, the integrand is $-\det^{\frac{1}{d-1}}$ and the convex constraint is positive semi-definiteness of a matrix field.

\vspace{4pt}


\vspace{4pt}


\vspace{4pt}

\noindent\textsc{Date:} \today{}.
\end{abstract}
\maketitle






\section{Introduction}

We will study (sequential) weak lower semi-continuity criteria for functionals of the form
\begin{equation}
u\mapsto \int_\Omega f(u(x))\d x
\end{equation}
  with respect  to weakly converging sequences $u_n\rightharpoonup u$ in $L^p$. With appropriate growth bounds for $f$, for instance $|f(v)|\le  C(1+|v|^p)$, we ask for which integrands $f$ we have this lower semi-continuity property. 
  
  If one considers all sequences in $L^p$ then we need $f$ to be convex. When we restrict the space of functions further then one can weaken the requirements on $f$. It  was shown by Morrey \cite{Morreyquaiconvexity}  that, considering only sequences of gradients, it suffices to require $f$ to be quasiconvex,
  that is, 
	\begin{equation}
f(\zeta)\le \frac{1}{|Q|}\int_{Q} f(\zeta+\nabla w(x))\d x
\end{equation}
for all $\zeta$ and all smooth $w$ with compact support on a cube $Q$. One can view this as a Jensen-type inequality but for gradients only. 
On a cube, a function is a gradient if and only if its curl is zero and thus we obtain
an equivalent condition by considering
weakly converging  sequences in $L^p$ that are curl-free, i.e.  $\curl u =0 $, and  thus assume that $f$ satisfies  
	\begin{equation}
f(\zeta)\le \frac{1}{|Q|}\int_{Q} f(\zeta+u (x))\d x
\end{equation}
for all $\zeta$ and all smooth curl-free $u$ with compact support on a cube $Q$.

In order to study other physical systems we can generalise these methods to study 
$\mathcal A$-free systems with $\mathcal A$ a  linear, homogeneous, differential operator with constant coefficients and, in most cases, constant rank. 
The $\mathcal A$-free approach was introduced in the work of Murat \cite{Muratcompensatedcompactness} and Tartar \cite{Tartarfrenchcompensated}, \cite{Tartarcompensated} studying compensated compactness, and has applications in
elasticity, plasticity, elasto-plasticity, electromagnetism etc.  Dacorogna \cite{DacorognaBookgenaralquasiconvexity} then established the corresponding notion of $\mathcal{A}$-quasiconvexity and showed it to be sufficient for lower semicontinuity, and a variant of this quasiconvexity condition was later shown by Fonseca and M\"uller \cite{FonescaMuelleraquasiconvexity} to be sufficient and necessary for lower semicontinuity.

We say that $\mathcal B$ is a  \emph{potential operator}  for the $\mathcal A$-free system if $\mathcal B$ is a linear homogeneous differential operator with constant coefficients, defined on the set of smooth periodic functions of zero mean, such that $\operatorname{im}\mathcal B=\operatorname{ker}\mathcal A$.
We know, for instance, that the gradient is the potential operator for curl-free systems.   With the recent work of Raita \cite{RaitaPotentials} one can derive  potentials in the general $\mathcal A$-free setting assuming $\mathcal A$ has constant rank. However, on general bounded domains and for general linear constraints, it is not the case that for every $\mathcal A$-free $u$ there exists a potential $\Phi$ such that $\mathcal B \Phi=u$. This unfortunately restricts the results of this paper to the torus (but see Remark~\ref{bdddomain} below).

Anyway, using  these potentials,   one establishes  a new definition for $\mathcal A$-quasiconvexity of the form 
	\begin{equation}
	\quad  \quad f(\zeta) \le  \int_{\T^d} f(\zeta +\mathcal{B}\Phi(x)) \d x
	\end{equation}
	for all $\zeta$  and $\phi \in C^\infty (\T^d)$ where $\T^d=\R^d/\Z^d$ is the $d$-dimensional torus with unit volume. 
	We will go into more detail later as we will use and develop the tools in \cite{RaitaPotentials} for our own situation.             

Many 
physical models fall into the constant rank $\mathcal A$-free framework, but with a codomain restricted to a convex set. 
There are examples  from compressible fluids studying the Euler, the Euler-Fourier, and relativistic Euler equations, models for rarefied gases (Boltzmann, discrete kinetic models, BGK), and other models for electromagnetic fields in a vacuum and the mass-momentum tensor in the Schr\"odinger equation.  
In these cases it makes sense to consider the problem of $\mathcal A$-quasiconvexity and (sequential) lower-semicontinuity only for physically relevant values inside such a convex set. Let $K\subset \R^N$ be a convex set with non-empty interior (i.e.\ of full dimension), then we define $K$-$\mathcal A$-quasiconvexity and $K$-$\mathcal A$-free sequences, in analogy to the standard definitions, by adding the constraint that $U\in K$ for almost every $x\in\T^d$.  We will define these terms more rigorously later on.  
We can now state our main result of this paper. 
 \begin{theorem}[Main Theorem]
	Let $p>d$ and
	 $f: K \to \Rbb$ be continuous and satisfy the growth bound  $|f(z)|\le C(1+|z|^r) $ for some $C>0$ and some $r< p$. 
	Then  
	\begin{equation}
	\liminf_{n\to \infty} \int_{\T^d} f(U_n) \d x \ge \int_{\T^d} f(U)\d x 
	\end{equation}
	for every sequence $(U_n)\in L^p(\T^d)$ of $K$-$\mathcal A$-free maps on $\Omega$ that converges weakly in $L^p$ to $U\in L^p(\T^d)$
	if and only if $f$ is $K$-$\mathcal {A}$-quasiconvex.
\end{theorem}

Note that there are interesting cases where the constraint does not define a  convex set: For instance, imposing local non-interpenetration of matter in nonlinar elasticity, one assumes that $\det \nabla u >0$ almost everywhere.  
This restriction models the assumption that under a deformation the matter will not change orientation or be compressed to a point, see \cite{Continontinterpentration} and  \cite{Rindlernoninter}  for more details. 

One potential use for these techniques would be in the area of convex integration. In \cite{lazloconvexeuler} as in other examples one has a functional used to control the sequence of sub-solutions and a convex set of values which each sub-solution can take. In \cite{lazloconvexeuler} lower semi-continuity is shown directly for the relevant functional.   

\subsection{Divergence-Free Positive Tensors}

One motivational result inspiring  this paper has been \cite{SerreDPTs}. 
For dimension $d\geq 2$ and the function $F(A)=(\det A)^{\frac{1}{d-1}}$, D.~Serre studied the specific functional  $\int_\Omega F(U(x))\d x$, where $\Omega$ is either a bounded convex domain or the torus, and $U: \Omega\to\R^{d\times d}$ a matrix field whose values are positive semi-definite. More specifically, he considered symmetric, divergence-free, positive semi-definite tensors defined as follows:  
\begin{definition}
	Let $\Omega$ be an open subset of $\R^d$, then a $DPT$ is defined as a locally integrable, divergence-free, positive symmetric tensor $x\mapsto U(x)$, that is, $U\in L^1_{loc}(\Omega;\R^{d\times d})$  with the properties that 
	\begin{equation}
	U(x)\in \mathrm{Sym}_d^+:=\{A\in\R^{d\times d}: \quad A^t=A, \quad A\geq0\}
	\end{equation}
	almost everywhere and $\diverg U=0$ (row-wise) in the sense of distributions.  
\end{definition}
For this functional on the torus a specific result of Serre was the following.
\begin{theorem}\label{theorem:Jensensinequality}
	Let the $DPT$ $x\mapsto U(x)$ be periodic over $\T^d$, with $U\in L^1(\T^d)$. Then $F(U)\in L^1(\T^d)$ and there holds 
	\begin{equation}
	\int_{\T^d} F(U(x))\d x \le F\left(\int_{\T^d} U(x)\d x \right). 
	\end{equation}
\end{theorem} 
This can be used to obtain a gain of integrability and can be applied to compressible fluid models to improve the a-priori estimates. These methods have been used recently in many cases, see \cite{SerreClassovpoission}, \cite{SerreMultidconservationlaws}, \cite{Serrecollionsmoleculardynamics} and  \cite{SerreSilvestermultiburgers}. 

For this specific $F$ and $U$, further questions were raised. In~\cite{SerreDPTs}, Serre asks whether or not this estimate could be improved to work for general $L^p$. Specifically, he asks the following:
\begin{problem}
Let $U\in L^p(\T^d;\mathrm{Sym}^+_d)$ and $\diverg U\in L^p(\T^d)$ with $1<p<d$. Defining $\tfrac{1}{p'}=\tfrac{1}{p}-\tfrac{1}{d}$, is it true that 
	\begin{equation}
	\det (U)^{\frac{1}{d}} \in L^{p'}(\T^d)?
	\end{equation} 
\end{problem}

In \cite{RossaSerreintegrability} it is shown that there are counterexamples disproving this statement and thus that there exist many $U$ in $L^p$ for $1<p<\frac{d}{d-1}$ such that $(\det U)^{\frac{1}{d-1}}\in L^1\setminus L^{1+\eps}$, for all $\eps >0$.

Another question, closely related to the results in this paper, is whether and for what $U\in L^p$ this functional is weakly upper semi-continuous.  In the recent work \cite{RossaSerresemiconvaty},  
the space
\begin{equation}
X_p:=\{A\in L^p(\T^d; \mathrm{Sym}^+_d)\colon \diverg A\in \mathcal M (\T^d;\R^d)\}
\end{equation}
is defined together with a weak topology such that $A_k\rightharpoonup A$ in $X_p$ if $A_k\rightharpoonup A$ weakly in $L^p$ and $\diverg A_k\stackrel{*}{\rightharpoonup} \diverg A$ in the sense of measures. Then, the following is proved in~\cite{RossaSerresemiconvaty}:
\begin{theorem}
	Let $p>\tfrac{d}{d-1}$ and $\{U_k\}_k\subset X_p$ be such that $U_k\rightharpoonup U$ in $X_p$. Then we have 
	\begin{equation}
	\limsup_k \int_{\T^d} (\det U_k(x))^{\frac{1}{d-1}}\d x \le \int_{\T^d} (\det U(x))^{\frac{1}{d-1}}\d x.
	\end{equation}
\end{theorem}
 Further, for the case of $p=\frac{d}{d-1}$, explicit counterexamples are provided. This completes the picture of upper semi-continuity for this specific functional and $U$ a $DPT$. 

We see that if we let $\mathcal A=\diverg$, then this is a form of $\mathcal A$-quasiconcavity but on a restricted set for $U$, which also has to be symmetric and positive definite. Thus, under strengthening of the assumption on the divergence to $\diverg U_n=0$, this is a specific case of our problem we have set in this paper.
In such a case, we can generate a potential for symmetric divergence-free matrices where for the potential $\Phi $ we have $\|\Phi\|_{W^{l,p}}\le C \|U\|_{L^p}$ for a suitable integer $l$.
 Thus we can apply our general theory to the specific case above and obtain results (although by our current techniques we need to require $p>d$, which is sharp for $d=2$). 


In Section \ref{sectio:Lowersemicontinuity} we will introduce notation and definitions so that we can state our main result (Theorem \ref{theorem:lscoffinfty}). 
We will introduce Young measures  and show that our main theorem can be reduced to showing that there exists appropriate sequences of functions that generate  homogeneous Young measures, similarly to the methods of Fonseca and M\"uller in \cite{FonescaMuelleraquasiconvexity}. 
We will introduce potential operators and potentials and justify an ellipticity bound to gain Sobolev control for all the derivatives (Theorem~\ref{generalpoincare}). 
Finally, in Section \ref{section:Maintheoremproof} we prove the main result by generating the sequences of functions that generate the homogeneous Young measures we need. In the appendix \ref{appendix} we calculate  a specific example of a Raita potential and a different proof for deriving the ellipticity control for the potential and that may be considered of independent interest and could be especially relevant to problems involving domains with a boundary.   


\section{Preliminaries}\label{sectio:Lowersemicontinuity}

We want to show lower semi-continuity of the functional
\begin{equation}
I(U):= \int_{\T^d} f(U(x)) \d x 
\end{equation}
 from $\mathcal{A}$-quasiconvexity of $f$.  Here, $U:\T^d \to \R^N$ and  $f:\R^N\to \R$. 
 By $\mathcal A$ we mean a $k$-homogeneous, linear, differential operator $\mathcal A: C^\infty(\T^d;\R^N)\to C^\infty(\T^d;\R^m)$, 
 \begin{equation}
 \mathcal A := \sum_{|\alpha|=k} \mathcal A^\alpha \partial_\alpha, 
 \end{equation}
 where $\mathcal A^\alpha\in \mathrm{Lin}(\R^N,\R^m)$. Here, we sum over all $d$-dimensional multi-indices $\alpha$ such that $|\alpha|=k$. We define the associated Fourier symbol map 
 \begin{equation}
 \mathcal A [\xi] := \sum_{|\alpha|=k} \xi^\alpha \mathcal A^\alpha 
 \end{equation}  
 for $\xi\in \R^d$. Here $\mathcal A$ is assumed to satisfy the constant rank property, i.e.\ there exists an $r\in \N$ such that $\mathrm{Rank} \mathcal A[\xi]=r $ for all $\xi\in \R^d\setminus \{0\}$.
 
 
   We will have extra convex restrictions on the set of values our functions can take. The usual form of $\mathcal {A}$-quasiconvexity as stated in \cite{FonescaMuelleraquasiconvexity} reads as follows:
\begin{definition}\label{def:Aquasiconvexityfonsecamueller}
	A function $f:\R^N\to \R$ is said to be 
	 $\mathcal {A}$-quasiconvex if
	\begin{equation}
	 f(\zeta)\le \int_{\T^d} f(\zeta+w(x))\d x
	\end{equation}
	for all $\zeta\in \R^N$ and for all $w\in C^\infty(\T^d;\R^N)$ such that $\mathcal{A}w=0$ and $\int_{\T^d} w(x)dx=0$.
\end{definition} 

(Recall that the torus has unit volume.) Let $K$ be a (not necessarily bounded) convex subset of $\R^N $ with non-empty interior (i.e.\ a convex set of full dimension, so to speak). We can add $K$ into the definitions of  $\mathcal {A}$-quasi convexity to obtain a convexly restricted form. For this, we assume there exists a potential operator for $\mathcal A$, i.e.\ a linear homogeneous differential operator $\mathcal B$ of order $l$, defined on smooth periodic functions of zero mean, with $\image \mathcal B=\ker \mathcal A$. Owing to~\cite{RaitaPotentials}, this will be the case if and only if $\mathcal A$ is of constant rank.  
\begin{definition}\label{def:ConvexAquasiconvexity}
	A function $f\colon K\to \R$ is $K$-$\mathcal {A}$-quasi convex if
	\begin{equation}\label{eq:PotentiaconvexAquasiconvexity}
	f(\zeta) \le  \int_{\T^d} f(\zeta +U(x)) \d x 
	\end{equation}
	for all 
	$U\in C^\infty(\T^d)$ with $\int_{\T^d}U(x)dx=0$ and $\zeta + U(x)\in K$ for almost every $x$ and $\mathcal{A} U=0$, or, equivalently, if
	\begin{equation}
	\quad  \quad f(\zeta) \le \int_{\T^d} f(\zeta +\mathcal{B}\Phi(x)) \d x
	\end{equation}
	for all $\Phi \in C^\infty (\T^d)$ with 
	$\zeta+\mathcal{B} \Phi(x)\in K$ for almost every  $x$.  
\end{definition}

\begin{remark}
	As in \cite{FonescaMuelleraquasiconvexity} Remark 3.3 (ii), if $f$ is upper semi-continuous and locally bounded above and $f(V)\le C(1+|V|^p)$ for all $V\in K$ for some $C>0$, then one can replace $C^\infty$ by $L^p$ in \eqref{eq:PotentiaconvexAquasiconvexity}.  
\end{remark}

We see that this is the same definition with potentials as in \cite{RaitaPotentials}  except we have asked for 
 the restriction to the set $K$.
\begin{definition}
We will define a function $U\in L^1_{loc}(\T^d)$ to be $K$-$\mathcal A$-free if  
 $U (x) \in K$ for almost every $x$ and $U$ is weakly $\mathcal A$-free, that is, $\mathcal A U=0$ in the sense of distributions. 
\end{definition}

%
%
%

Here we want to prove lower semi-continuity of an $f$ that is $K$-$\mathcal {A}$-quasiconvex for weakly converging sequences of functions $U_n$ which are $K$-$\mathcal A$-free. 
\begin{theorem}[Main Theorem]
\label{theorem:lscoffinfty}
	Let $p>d$ and
	 $f: K \to \Rbb$ be continuous and satisfy the growth bound  $|f(z)|\le C(1+|z|^r) $ for some $C>0$ and some $r< p$. 
	Then  
	\begin{equation}
	\liminf_{n\to \infty} \int_{\T^d} f(U_n) \d x \ge \int_{\T^d} f(U)\d x 
	\end{equation}
	for every sequence $(U_n)\in L^p(\T^d)$ of $K$-$\mathcal A$-free maps on $\Omega$ that converges weakly in $L^p$ to $U\in L^p(\T^d)$
	if and only if $f$ is $K$-$\mathcal {A}$-quasiconvex.
\end{theorem} 


\begin{remark}
		Consider a Carath\'eodory integrand with explicit dependence on the domain, i.e., $f:\T^d\times K \to \R  $ measurable in the first and continuous in the second variable. Under our previous assumptions on the second variable one can even treat such integrands.  
\end{remark}


We will be using Young measures (parametrised probability measures) $\nu_x$ for $x\in \T^d$. These are useful objects that hold more information about the limit of weakly converging sequences of functions and a natural tool when discussing lower semicontinuity. We will state a version of the fundamental theorem of  Young measures which demonstrates this link.  Below is  a modified version of the theorem by Fonseca and M\"uller in \cite{FonescaMuelleraquasiconvexity}, see also\cite{RaitaPotentials}. 

\begin{theorem}[Fundamental Theorem on Young Measures] \label{theorem:Fundamentalt theroem of young measures}
	Let $\Omega \subset \R^d$ be a measurable set of finite measure and let  $\{z_n \}$ be a sequence of measurable functions $\Omega\to\R^N$ such that the sequence does not lose mass at infinity, i.e., $\lim_{M\to\infty} \sup_n |\{|z_n|\ge M \}|=0.$ 
	 Then there exist a subsequence (not relabelled) and a weakly*-measurable map\footnote{This means that $x\mapsto\int_{\R^N} f(z)d\nu_x(z)$ is measurable for every $f\in C_0(\R^N)$.} $\nu:\Omega\to \mathcal P(\R^N)$   such that the following hold:
	\begin{enumerate}
		\item let $D\subset \R^N$ be a compact subset then $\supp \nu_x \subset D$ for almost every $x\in \Omega$ if and only if  
		$\dist (z_n,D) \to 0 $ in measure;
		\item if $f:\Omega \times \R^N \to \R$ is a normal integrand  (Borel measurable in the first and lower semicontinuous in the second variable), bounded from below, then 
		\begin{equation}
		\liminf_{n\to \infty} \int_\Omega f(x,z_n(x))\d x \ge \int_\Omega \bar{f} (x) \d x \quad \text{where}\quad \bar{f} (x):= \langle \nu_x, f(x,\cdot  )\rangle =\int_{\R^N} f(x,y)\d {\nu_x(y)};
		\end{equation} 
		
		\item  If $f$ is Carath\'eodory ($f$ and $-f$ are normal integrands), then 
		\begin{equation}
		\lim_{n\to\infty } \int_{\Omega} f(x,z_n(x))\d x =\int_{\Omega} \bar f (x) \d x <\infty 
		\end{equation} 
		if  $\{f(z_n({\cdot})) \}$ is equi-integrable. In this case $f(\cdot, z_n(\cdot ))\rightharpoonup \bar f $ in $L^1(\Omega)$. \label{pointmeasurelimitforf}
 	\end{enumerate}  
\end{theorem}


Here $\nu_x$ is the Young measure generated by the sequence $\{z_n \}$. The Young measure is {\em homogeneous} if there is a Radon measure $\nu\in \mathcal M(\R^N)$ such that $\nu_x=\nu$ for almost every $x\in \Omega$. The growth bound  $|f(z)|\le C(1+|z|^r) $ for  $C>0$, $r< p$  is assumed so that   $\{f(z_n) \}$ is equi-integrable. 
If $\{z_n \}$ is bounded in $L^p$, $f$ is continuous and we assume the growth bound then $f(z_n)\rightharpoonup \bar f$ in $L^{\frac{p}{r}}$. If $\{z_n \}$  is equi-integrable, then letting $f=\mathrm{id} $ we see that $z_n\rightharpoonup \bar z$ in $L^1(\Omega)$ where $\bar z (x):= \langle \nu_x,\mathrm{id} \rangle$.   


\begin{theorem}\label{theorem:youngmeasurejensenp}
	Let $p>d$ and 
	$\{U_n\}$ be a weakly converging sequence in $L^p(\T^d)$ that is $K$-$\mathcal A$-free  
	and generates the Young measure $\nu=\{\nu_x\}_{x\in \T^d}$. 
	Then for almost every $a\in \T^d$ there exists a 
	sequence $\bar U_n:\T^d\to\R^N$ that is $K$-$\mathcal A$-free, has mean $\langle\nu_a,\id\rangle$, and generates the homogeneous Young measure $\nu_a$.
\end{theorem}


\begin{lemma}
	To prove the sufficiency part of Theorem \ref{theorem:lscoffinfty} 
	it is enough to show Theorem \ref{theorem:youngmeasurejensenp}. 
\end{lemma}
\begin{proof}
	Let $U_n$ be a sequence that is $K$-$\mathcal A$-free, 
	converges weakly in $L^p$ to $U$, and generates the Young measure $\nu_x$. Let $a\in\T^d$ be such that the conclusion of Theorem~\ref{theorem:youngmeasurejensenp} holds at $a$.  
	From Theorem~\ref{theorem:youngmeasurejensenp} 
	there exists a weakly converging sequence $\bar U_n$ that is $K$-$\mathcal A$-free 
	and generates the homogeneous Young measure $\nu_a$. As $\{\bar{U_n}\}$ is bounded in $L^p$ and $f$ grows slower than the $p$-th power, the sequence $\{f(\bar U_n)\}$ is equi-integrable, so we may apply the Fundamental Theorem~\ref{theorem:Fundamentalt theroem of young measures}(3) to find that 
	\begin{align}
	\langle \nu_a,f\rangle& =\int_{\T^d}\langle \nu_a,f\rangle \d x = \lim_{n\to\infty} \int_{\T^d} f(\bar U_n)\d x \\&\ge \lim_{n\to\infty} f \left(U(a) \right) = f(\langle \nu_a, \id \rangle).
	\end{align}  
	Here we used the homogeneity of the Young measure for the first equality, the properties of $f$ to see that $\{f(\bar U_n)\}$ is equi-integrable and Theorem~\ref{theorem:Fundamentalt theroem of young measures} part \eqref{pointmeasurelimitforf} for the second 
	 equality, 
	and finally definition \ref{def:ConvexAquasiconvexity} and the remark thereafter for the middle inequality. 
	 Thus, we see that $\langle \nu_x,f\rangle\ge f(\langle \nu_x, \id \rangle)$ for almost every $x$ and so 
	\begin{equation}
	\liminf_{n\to \infty} \int_{\T^d} f( U_n) \d x =\int_{\T^d}\langle \nu_x,f \rangle \d x \ge \int_{\T^d} f(\langle \nu_x,\id \rangle) \d x =\int_{\T^d} f( U)\d x, 
	\end{equation} 
	 where again we used Theorem~\ref{theorem:Fundamentalt theroem of young measures} part \eqref{pointmeasurelimitforf} for the first and last equalities. 
\end{proof}

 Were it not for the pointwise convex constraint, we could use the theory of Fonseca and M\"uller \cite{FonescaMuelleraquasiconvexity} to prove Theorem~\ref{theorem:youngmeasurejensenp}. 
 However, in order to generate the appropriate Young measure, they use a projection onto $\mathcal{A}$ free elements 
which is formulated using Fourier techniques. This could destroy the convex constraint on the sequence and thus invalidate Definition~\ref{def:ConvexAquasiconvexity}. 
Instead of using a (nonlocal) projection operator, we will focus on potential operators (which are local).


We will define a sequence of functions that need  to be smoothly cut off, and if this is naively done, it will  
invalidate the $\mathcal A$-free condition. In order to keep the sequence $\mathcal A$-free, we will perform the cut-off at the level of the potential functions and then apply the potential operator to  ensure we remain   $\mathcal A$-free. 

If this is done, as in~\cite{FonescaMuelleraquasiconvexity}, by means of nonlocal Fourier projectors, however, we lose any pointwise control in the cutoff region, and may thus violate the pointwise constraint $U\in K$. Instead, we will focus on 
potential operators as they 
  act locally, and thus we can perform estimates uniformly to keep the convex constraint.  


 From the work of Raita \cite{RaitaPotentials} for any linear, homogeneous differential operator $\mathcal A $ of constant rank and constant coefficients, there exists another such operator $\mathcal B$ of order, say, $l$, such that $\ker \mathcal A [\xi] =\image \mathcal B [\xi] $ for all $\xi\in \R^d\setminus\{0\} $.   Thus, on the torus, for any smooth function $U$ with $\mathcal{A} U=0$ and zero mean, we can find a smooth potential $\Phi$ such that $\mathcal{B} \Phi = U$. 

In order to use the construction in \cite{RaitaPotentials} for our method, we need the property that $\|\Phi\|_{W^{l,p}}\le C \|U\|_{L^p}$ for a suitable choice of $\Phi$. Using Fourier techniques from \cite{RaitaPotentials}, if $\mathcal A U =0$ in $\T^d$ and $\hat U (0)=0$ then 
\begin{equation}
\Phi(x) = \sum_{\xi \neq 0} \mathcal B^\dagger (\xi) \hat U(\xi) \mathrm{e}^{2\pi \mathrm{i} x\cdot \xi },
\end{equation}
where $\mathcal B ^\dagger$ is defined in \cite{RaitaPotentials} and the appendix \ref{appendix}. With a simple use of Fourier multiplier theorems from \cite{FonescaMuelleraquasiconvexity} we obtain that that $\|\Phi\|_{W^{l,p}}\le C \|U\|_{L^p}$.


\begin{theorem}\label{generalpoincare}
	Let $1<p<\infty$ and let $\mathcal A$ and $ \mathcal B$  be linear, homogeneous, differential operators, with constant rank and constant coefficients and assume that 
	\begin{equation}
	\ker \mathcal A[\xi]= \image \mathcal B [\xi] \quad \mathrm{and }\quad \ker \mathcal B[\xi]= \image \mathcal G [\xi] 
	\end{equation}
	for all $\xi\in \R^d\setminus \{0\}$. Then there exists a constant $C$ such that for all
	$U\in L^p(\T^d)$ 
	such that $\mathcal A U=0$ and $\int_{\T^d}U \d x =0$, there exists a 
	$\Phi\in W^{l,p}(\T^d)$ 
	(where $l=2k$ and $k$ is the homogeneity of $\mathcal A$) 
	such that
	 $\|\Phi\|_{W^{l,p}}\le C\|U\|_{L^p}$.
\end{theorem}

An explicit example (for  $\mathcal A=\diverg$ in two dimensions) is given in Appendix \ref{exaple:RaitaDivergence}. 

\begin{remark}\label{bdddomain}
\begin{enumerate}
\item Only on the torus is it true that for every $\mathcal A$-free tensor $U$ there exists a potential $\Phi$ such that $\mathcal B\Phi=U$. On a bounded domain or on $\R^d$, this is no longer the case: Consider the constraint $\Delta U=0$, which satisfies the constant rank condition. It is not difficult to compute $\mathcal B=0$ for the corresponding potential operator obtained from~\cite{RaitaPotentials}. While, by Liouville's Theorem, indeed every harmonic function on the torus is identically zero, this is of course not the case on a bounded domain, or on the whole space. Since our method relies heavily on the existence of such a potential, it is, in the stated generality, restricted to periodic boundary conditions.
\item Let us, however, indicate two ways to make statements in the spirit of Theorem~\ref{theorem:lscoffinfty} on certain domains. First, on any open set $\Omega\subset\R^d$, if we \emph{assume} the sequence to be obtained from a potential, then an analogous result to Theorem~\ref{theorem:lscoffinfty} holds. This amounts to considering the functional $\Phi\mapsto\int_{\Omega}f(\mathcal B \Phi)\dd x$ instead of $U\mapsto\int_{\T^d}f(U)\dd x$. Such functionals were considered recently in~\cite{CG2020}.
\item For specific constraints $\mathcal A$, we can obtain a potential by Poincar\'e's Lemma, as long as the domain is contractible. To demonstrate this, let us consider the case of a symmetric divergence-free matrix field $U:\Omega\to\R^{2\times 2}$. Writing \begin{equation*}
    U=\begin{pmatrix} u & v\\ 
    v & w
    \end{pmatrix},
    \end{equation*}
we see that $U$ is symmetric and divergence-free if and only there exist $\phi,\psi: \Omega\to\R$ such that 
\begin{equation*}
    (u,v)=(-\partial_2\phi,\partial_1\phi),\quad (v,w)=(-\partial_2\psi,\partial_1\psi),
\end{equation*}
which in turn is the case if and only if $\partial_1\phi=-\partial_2\psi$, which holds if and only if there exists $\Phi:\Omega\to\R$ such that $(\phi,\psi)=(-\partial_2\Phi,\partial_1\Phi)$.

In summary, $U$ is symmetric and divergence-free if and only if there exists $\Phi$ such that 
\begin{equation*}
    u=\partial_2^2\Phi,\quad v=-\partial_1\partial_2\Phi,\quad w=\partial_1^2\Phi.
    \end{equation*}
Similar, but more tedious computations work in higher dimensions.

We have thus found a potential that allows us to prove Theorem~\ref{theorem:lscoffinfty} on any contractible domain for the special case of DPTs.
	\end{enumerate}
\end{remark}

\section{Proof of Main Theorem}\label{section:Maintheoremproof}


Having developed all the tools needed, we can now prove the main theorem. 

\begin{proof}[Proof of Theorem \ref{theorem:lscoffinfty}]

\emph{Necessity.} The easier part of the proof is the necessity of $K$-$\mathcal A$-quasiconvexity for lower semicontinuity. So suppose
\begin{equation}
	\liminf_{n\to \infty} \int_{\T^d} f(U_n) \d x \ge \int_{\T^d} f(U)\d x 
	\end{equation}
	for every sequence $(U_n)\in L^p(\T^d)$ of $K$-$\mathcal A$-free maps on $\Omega$ that converges weakly in $L^p$ to $U\in L^p(\T^d)$. We need to show that $f$ is $K$-$\mathcal A$-quasiconvex. 
	
	To this end, let $\zeta\in\R^N$ and $U\in C^\infty(\T^d)$ with zero mean, such that $\zeta+U\in K$ almost everywhere. Set
	\begin{equation}
	U_n(x):=\zeta+U(nx)   
	    \end{equation}
	which converges weakly to the constant $\zeta$ and satisfies the requirements for lower semicontinuity. Thus, using a simple variable transformation,
	\begin{equation}
	    \int_{\T^d}f(\zeta+U(x))\dd x=\liminf_{n\to\infty}\int_{\T^d}f(U_n(x))\dd x\geq f(\zeta),
	    \end{equation}
	    as claimed.

\emph{Sufficiency.} It suffices to prove Theorem~\ref{theorem:youngmeasurejensenp}.

First, we can assume that $U_n$ is bounded away uniformly from $\partial K$. Indeed, take $Y\in \mathrm{Int}(K)$ (which is non-empty by assumption) and define 
\begin{equation}
V_{n,\delta}:= (1-\delta)(U_n-Y)+Y.
\end{equation}
 This has the effect of shrinking the codomain of $U_n$  from $K$ into  
 $K_{\delta}$  where 
 \begin{equation}
K_{\delta}:=\{x\in K:\dist(x,\partial K)\geq\delta \}
\end{equation}
for some ${\delta}>0$.
  We see that $\mathcal A V_{n,\delta}=(1-\delta)\mathcal A U_n=0$ for all $n$, $V_{n,\delta}\in L^p$ and by using \cite[Proposition 2.4]{FonescaMuelleraquasiconvexity} we see that $V_{n,\delta}$ will generate (in the parameter $n$) a Young measure that converges, as $\delta\to0$, weakly* in\footnote{We denote by $L_w^\infty(\Tbb^d;\mathcal M(\Rbb^N))$ the space of weakly*-measurable maps from $\Tbb^d$ to the space of finite measures on $\R^N$. It is dual to $L^1(\Tbb^d; C_0(\Rbb^N))$.} $L_w^\infty(\Tbb^d;\mathcal M(\Rbb^N))$ to the Young measure generated by $U_n$. Using a diagonal argument, we see that Theorem~\ref{theorem:youngmeasurejensenp} is true if it is true for all sequences that are bounded away from $\partial K$.  
   
Secondly, we may assume without loss of generality that $\int_{\T^d}U_n \d x = \int_{\T^d}U \d x$ for all $n\in\N$: Otherwise, consider instead the sequence
\begin{equation*}
W_n:=U_n- \int_{\T^d}(U_n-U) \d x
\end{equation*}
and note that $\int_{\T^d}W_n \d x=\int_{\T^d}U \d x$ for all $n\in\N$, that by the weak convergence the integral of $U_n-U$ converges to zero and therefore $W_n$ and $U_n$ generate the same Young measure, and that $W_n$ will take values in the interior of $K$ for sufficiently large $n$. Also, $W_n$ remains $\mathcal A$-free. Therefore, if $(W_n)_{n\in\N}$ satisfies the conclusion of Theorem~\ref{theorem:youngmeasurejensenp}, then so will $(U_n)_{n\in\N}$.




Let now $\mathcal{I}$ be a countable dense subset of $L^1(\T^d)$ and similarly, $\mathcal{C}$ be a countable dense subset of $C_0(\R^N)$. We know that if $U_n$ generates the Young measure $\nu_x$, then for any $g\in \mathcal{C}$, $g(U_n)\stackrel{\star}{\rightharpoonup}\langle \nu_x,g\rangle$ in $L^\infty(\Omega)$.
Let $\Omega_0$ be the set of all Lebesgue points $a\in \T^d$ of $U$ which are at the same time Lebesgue points for the functions 
\begin{equation}
x\mapsto \int_{\R^N} |\xi|^p \d \nu_x(\xi) 
\end{equation}   
and $x\mapsto \langle \nu_{x}, g \rangle$ for $g\in \mathcal{C}$, so that in particular
\begin{equation}
\lim_{R\to 0} \int_{\T^d} |\langle \nu_{a+Rx},g \rangle -\langle \nu_a ,g\rangle|\d x=0.
\end{equation}


Let $\Phi_n$ be the potential for the sequence $U_n-U$ (which has mean zero!), thus, $\mathcal{B} \Phi_n(x)=U_n(x)-U(x)$.

Identifying the torus $\T^d$ with a cube $Q$, let $\chi_j$ be a sequence of positive smooth cut-off functions in $C_c^\infty(Q)$ converging monotonically up to 1. 
Then for a fixed $a\in \Omega_0$ we define a new family $U_{j,R,n}$ ($R>0$, $j,n\in\mathbb{N}$) by 
\begin{equation}
U_{j,R,n}(x):= \mathcal{B}[\chi_{j}(x)R^{-l}\Phi_n(a+Rx)]+U(a+Rx) 
\end{equation}
for $x\in \T^d$. 

We have assumed that $U_n-U\rightharpoonup 0$ in $L^p$, so by linearity of $\mathcal B$ and the compact embedding $W^{1,p}\hookrightarrow C^{0,\eps}$ for some $\eps>0$   we obtain uniform convergence of $\partial_{\beta}\Phi_n\to 0$ for $|\beta|\le  l-1$. 

We clearly see that, as $\mathcal{B}$ maps to the kernel of $\mathcal{A}$, 
\begin{equation}
\mathcal{A} (U_{j,R,n})= \mathcal{A}(\mathcal{B}[\chi_j(x)R^{-l}\Phi_n(a+Rx)]+U(a+Rx))=0
\end{equation}
for all $j,R,n$. 

We can expand the expression to see that
\begin{align}
U_{j,R,n}(x)=& \mathcal{B}[\chi_j(x)R^{-l}\Phi_n(a+Rx)]+U(a+Rx)\\
=&\chi_j(x)\left(U_n(a+Rx) -U(a+Rx)  \right)+U(a+Rx)\\
+& \sum_{|\alpha|+|\beta|=l, |\beta|>1} R^{|\alpha|-l} \partial_{\alpha}\Phi_n(a+Rx) B^{(\alpha,\beta)} \partial_\beta \chi_j(x).
\end{align}
One notices that the $R^{-l}$ was added so that it scales with the $R$ gained from each differentiation, as the operator $\mathcal{B}$ is homogeneous of degree $l$.

We will now show that the sequence $U_{j,R,n}\in L^p(\T^d;\R^d)$ generates the homogeneous Young measure $\nu_a$.  For all $\psi\in \mathcal{I}$ and $g\in \mathcal{C}$,  
\begin{align}
\lim_{j\to\infty}&\lim_{R\to 0}\lim_{n\to\infty}\int_{\T^d} \psi(x)g(U_{j,R,n}(x))\d x\\
=&\lim_{j\to\infty}\lim_{R\to 0}\lim_{n\to\infty}\int_{\T^d} \psi(x)g\Big[\chi_j(x)\left(U_n(a+Rx) - U(a+Rx)   \right)+U(a+Rx)\\
&\hspace{2cm}+ \sum_{|\alpha|+|\beta|=l} R^{|\alpha|-l} \partial_{\alpha}\Phi_n(a+Rx) B^{(\alpha,\beta)} \partial_\beta \chi_j(x)\Big]\d x\\
=&\lim_{j\to\infty}\lim_{R\to 0}\int_{\T^d} \psi(x)\int g\left[\chi_j(x)\left(\xi - \langle \id,  \nu_{a+Rx}\rangle    \right)+\langle \id, \nu_{a+Rx}\rangle\right]\d \nu_{a+Rx}(\xi)\d x\\
=&\lim_{j\to\infty}\int_{\T^d} \psi(x)\int g\left[\chi_j(x)\left(\xi -\langle \id,  \nu_{a}\rangle   \right)+\langle \id, \nu_{a}\rangle\right]\d \nu_{a}(\xi)\d x\\
=&\langle \nu_a,g\rangle\int_{\T^d} \psi(x)\d x,
\end{align}
where we used the uniform convergence of $\partial_\alpha\Phi_n$, the property that $U_n$ generates $\nu$, and the Lebesgue point property of $a$.

We finally need to show that $U_{j,R,n}(x)\in K$ for almost every $x$. We see that 
\begin{multline}
\left|U_{j,R,n}(x)-\left[\chi_j(x) U_n(a+Rx) +(1-\chi_j(x)) U(a+Rx)\right]\right|\\
\leq\left |\sum_{|\alpha|+|\beta|=l} R^{|\alpha|-l} \partial_{\alpha}\Phi_n(a+Rx) B^{(\alpha,\beta)} \partial_\beta \chi_j(x)\right |.
\end{multline}
The right hand side converges to zero uniformly (when the limit is taken in the order $n\to\infty$, then $R\to0$, then $j\to\infty$), and the second two terms on the left hand side form a convex combination of functions with values in $K$ and distance at least $\delta_n$ from $\partial K$. This ensures that $U_{j,R,n}(x)\in K$ almost everywhere for sufficiently large $n$.  

We can now use a diagonalisation procedure (choosing $R=R(n)$ and $j=j(n)$ appropriately) to find a subsequence $V_n$ (relabelled with index $n$) with all the properties shown above. This establishes Theorem~\ref{theorem:youngmeasurejensenp} and thus Theorem~\ref{theorem:lscoffinfty}.

\end{proof}

\appendix

\section{Potential for Divergence} \label{appendix}

Here we give an explicit example of a potential operator for the divergence of a two-dimensional vector field (though this also works in higher dimensions) in order to better illustrate the techniques from \cite{RaitaPotentials}.

\begin{example}\label{exaple:RaitaDivergence}
	Let $\mathcal A = \diverg$ in two dimensions. Then $\mathcal A[\xi]=\xi^T$ and $\mathcal A^*[\xi]=\bar \xi$ and so 
	\begin{equation}
	\mathcal A[\xi]\mathcal A^*[\xi]=\begin{pmatrix}
	\xi_1\xi_1 &\xi_1\xi_2\\
	\xi_1\xi_2 & \xi_2\xi_2
	\end{pmatrix}.
	\end{equation}
	We want to calculate the characteristic  polynomial of this matrix so we must look at 
	\begin{equation}
	\det [\mathcal A[\xi]\mathcal A^*[\xi]-\lambda \, \mathrm{Id}]= \det \begin{bmatrix}
	\xi_1\xi_1 -\lambda&\xi_1\xi_2\\
	\xi_1\xi_2 & \xi_2\xi_2-\lambda 
	\end{bmatrix}=\lambda^2-|\xi|^2\lambda=(a^a_0)\lambda^2+(a^a_1)\lambda.
	\end{equation}  
	From this we see that $r^a=1$ and
	\begin{equation}
	\mathcal A^\dagger [\xi]=\frac{1}{|\xi|^2}\xi[a^a_0(\mathcal A[\xi]\mathcal A^*[\xi])^0]=\frac{\xi}{|\xi|^2}\mathrm{Id}
	\end{equation}
	and thus 
	\begin{equation}
	\mathcal L [\xi]=|\xi^2|[\mathrm{Id}-\frac{\xi}{|\xi|^2}\mathrm{Id}\xi^T]=(|\xi|^2\mathrm{Id}-\xi\otimes \xi).
	\end{equation}
	Similarly, we can apply the same method to $\mathcal L$. We see that 
	\begin{equation}
	\det[\mathcal L \mathcal L^* -\lambda \mathrm{Id}]=-(\lambda^2+|\xi|^4\lambda) =a^l_0\lambda^2 + a^l_1
	\end{equation} 
	and so $r^l=1$ and we obtain that
	\begin{equation}
	\mathcal L^\dagger[\xi] =-\frac{1}{|\xi|^4}\mathcal L^*[\xi](-1)=\frac{1}{|\xi|^4}(|\xi|^2\mathrm{Id} -\xi\otimes \xi)
	\end{equation}
	and thus
	\begin{align}
	\mathcal{G}[\xi] &=|\xi|^4[\mathrm{Id}-\frac{1}{|\xi|^4}(|\xi|^2\mathrm{Id} -\xi\otimes \xi)(|\xi|^2\mathrm{Id} -\xi\otimes \xi)]\\
	&=[2|\xi|^2 \xi\otimes \xi -(\xi\otimes \xi)^2].
	\end{align}
	We shall consider the system
	\begin{align}
	\mathcal L \Phi&=U\\
	\mathcal G \Phi&=0 
	\end{align} 
	and see that 
	\begin{align}
	\mathcal L[\xi] \hat\Phi(\xi)=(|\xi|^2\mathrm{Id}-\xi\otimes \xi) \hat\Phi(\xi) &=\hat U(\xi)\\
	\mathcal G[\xi] \hat\Phi(\xi)=[2|\xi|^2 \xi\otimes \xi -(\xi\otimes \xi)^2] \hat\Phi(\xi)&=0.
	\end{align}
	We can manipulate $\mathcal G[\xi] \hat \Phi [\xi]$ and see that 
	\begin{equation}
	\mathcal G[\xi] \hat\Phi(\xi)=(\xi\otimes \xi)[|\xi|^2\mathrm{Id}+|\xi|^2\mathrm{Id}  -(\xi\otimes \xi)] \hat\Phi(\xi)=0
	\end{equation}
	and thus 
	\begin{align}
	(\xi\otimes \xi)|\xi|^2\mathrm{Id}\hat\Phi(\xi)&=-(\xi\otimes \xi)[|\xi|^2\mathrm{Id}  -(\xi\otimes \xi)] \hat\Phi(\xi)\\
	&=(\xi\otimes \xi)(-\mathcal L[\xi]) \hat\Phi[\xi]= -(\xi\otimes \xi)\hat U(\xi).
	\end{align}
	We can rearrange this and obtain that 
	\begin{equation}
	\hat\Phi(\xi)=-|\xi|^{-2}\hat U(\xi).
	\end{equation} 
	We thus have an elliptic operator relating $U$ and $\Phi$ and can apply Calder\'on-Zygmund theory to the singular integral kernel generated from this symbol. Assuming a Lipschitz domain, we obtain 
	for $1<p<\infty$ that $\|\Phi\|_{W^{2,p}}\le C \|U\|_{L^p}$ as we needed. 
\end{example}

In order to use the construction in \cite{RaitaPotentials} we used the  property that $\|\Phi\|_{W^{l,p}}\le C \|U\|_{L^p}$ for a suitable choice of $\Phi$. Here we show another method to do this where we will need to find another condition on $\Phi$ that does not interfere with the property $\mathcal B \Phi =U$, yet 
creates an elliptic system that will give us the above bound.
We see that as $\mathcal B$ is another linear homogeneous differential operator with constant rank and constant coefficients, we can apply the theory of Raita again to find a      linear homogeneous differential operator with constant rank and constant coefficients $\mathcal G$ such that
\begin{equation}
\ker \mathcal B=\image \mathcal G
\end{equation}
and so this will characterise the kernel of $\mathcal B$. Now, any $\Phi$  can be split into two parts  $\Phi=\Psi +\mathcal G h$, where $\Psi$ is orthogonal to the kernel of $\mathcal B$. If we impose the condition $\mathcal G h=0$,  this will force  $\Phi=\Psi$ and thus $\Phi$ will be orthogonal to the kernel of $\mathcal B$. Further, by construction,  $\mathcal G$ is self-adjoint and so for all $h\in C^\infty_c$ 
we have
\begin{equation}
0=\langle \mathcal G h, \Phi\rangle= \langle  h,\mathcal G  \Phi\rangle.
\end{equation}
Thus, having imposed $\mathcal G h=0$, we obtain   the condition 
$\mathcal G \Phi=0$. 
This suggests that $\mathcal G \Phi=0$ would be a suitable condition to select only the $\Phi$ that has no component in the kernel of $\mathcal B$. Indeed we will show that this is enough to 
control the full derivative of $\Phi$ by $U$ in $L^p$ for $1<p<\infty$. 

Further, as $\mathcal B$ and $\mathcal G$ are self-adjoint, the domain and target spaces of the respective Fourier symbols are the same, say $W$, while the target space of the symbol of $\mathcal A$ is called $V$; see the following diagram: 
\begin{equation}
\cdots \underset{h}{W}\xrightarrow[]{\mathcal{G}}  \underset{\Phi, \mathcal G h=\Phi}{W}\xrightarrow[]{\mathcal B} \underset{U, \mathcal B \Phi=U}{W}\xrightarrow[]{\mathcal A} \underset{\mathcal A U}{V}
\end{equation}   

\begin{theorem}
	Let $1<p<\infty$ and let $\mathcal A, \mathcal B$ and $\mathcal G$ be linear, homogeneous, differential operators, with constant rank and constant coefficients and assume that 
	\begin{equation}
	\ker \mathcal A[\xi]= \image \mathcal B [\xi] \quad \mathrm{and }\quad \ker \mathcal B[\xi]= \image \mathcal G [\xi] 
	\end{equation}
	for all $\xi\in \R^d\setminus \{0\}$. Then there exists a constant $C$ such that for all
	$U\in L^p(\T^d)$ 
	such that $\mathcal A U=0$ and $\int_{\T^d}U \d x =0$, there exists a 
	$\Phi\in W^{l,p}(\T^d)$ 
	(where $l=2k$ and $k$ is the homogeneity of $\mathcal A$) such that
	\begin{align}
	\mathcal L \Phi&=U\\
	\mathcal G \Phi&=0,
	\end{align}	  
	and such that
	$\|\Phi\|_{W^{l,p}}\le C\|U\|_{L^p}$.
\end{theorem}

\begin{remark}
	For the case of $DPT$s, it is possible to construct the potential operator ``by hand'' via successive application of Poincar{\'e}'s Lemma, see Remark~\ref{bdddomain}. 
\end{remark}	

When $U\in C^\infty$, we observe from Theorem 1 in \cite{RaitaPotentials}  that from $\mathcal A$ we have the existence of $\mathcal B $ and $\mathcal G $ and further from Lemma 5 we see that $\Phi\in C^\infty(\T^d) $ exits. Here we just have to prove that adding the extra condition $\mathcal G \Phi=0$ gives the bound  $\|\Phi\|_{W^{l,p}}\le C\|U\|_{L^p}$, and the extension to $L^p$ then follows immediately from an easy approximation. 

Before we begin the proof, let us fix some notation. For a matrix $M\in \R^{n\times m}$ we define its pseudo-inverse   $M^\dagger\in \R^{m\times n}$    as the unique matrix determined by the relations 
\begin{equation}
M M^\dagger M=M, \quad M^\dagger M M^\dagger=M^\dagger,\quad (M^\dagger M) ^* =M^\dagger M, \quad (M M^\dagger) ^* =M M^\dagger,
\end{equation}
where $M^*$ is the adjoint of $M$. Further, from  \cite{Decellgeneralisedinverse}, for a matrix $M\in \R^{n\times m}$ denote by 
\begin{equation}
p(\lambda):= (-1)^n(a_0\lambda^n+a_1 \lambda^{n-1}+\cdots +a_n )
\end{equation}
as the characteristic polynomial of $MM^*$, where $a_0=1$. Define $r=\max \{ j\in \N \colon a_j>0 \}$. Then, if $r=0$, we have that $M^\dagger =0$; else 
\begin{equation}
M^\dagger=-a_r^{-1} M^* [a_0(MM^*)^{r-1} +a_1 (MM^*)^{r-2}+\cdots +a_{r-1}\mathrm{Id}_{n\times n}].
\end{equation}

\begin{proof}
	
	We can apply the theory in \cite{RaitaPotentials} to obtain the following formulas for $\mathcal{B}[\xi]$ and $\mathcal{G}[\xi]$:
	\begin{equation}
	\mathcal B[\xi]=a^a_r(\xi)[\mathrm{Id}-\mathcal A^\dagger[\xi]\mathcal{A}[\xi]],
	\end{equation}  
	where $a^a_r(\xi)$ is the $r^{th}$ term of the characteristic polynomial of $\mathcal A[\xi]\mathcal A^*[\xi]$, and similarly 
	\begin{equation}
	\mathcal G[\xi]=a^l_r(\xi)[\mathrm{Id}-\mathcal B^\dagger[\xi]\mathcal{L}[\xi]],
	\end{equation}  
	where $a^l_r(\xi)$ is the $r^{th}$ term of the characteristic polynomial of $\mathcal B[\xi]\mathcal B^*[\xi]$.
	We want to represent $\mathcal G[\xi]$ in terms of $\mathcal A[\xi]$. Firstly, we will expand $\mathcal B^\dagger[\xi]$ and so obtain
	\begin{multline}
	\mathcal B^\dagger[\xi]=-(a^l_r)^{-1}(\xi) \mathcal B^*[\xi]\big[a^l_0(\xi)(\mathcal B[\xi]\mathcal B^*[\xi])^{r-1}+\\a^l_1(\xi)(\mathcal B[\xi]\mathcal B^*[\xi])^{r-2}+\cdots +a^l_{r-1}(\xi)\mathrm{Id}\big]
	\end{multline}
	and so $a^l_r(\xi)\mathcal B^\dagger[\xi]\mathcal{L}[\xi]$ becomes 
	\begin{align}
	&- \mathcal B^*[\xi][a^l_0(\xi)(\mathcal B[\xi]\mathcal B^*[\xi])^{r-1}+a^l_1(\xi)(\mathcal B[\xi]\mathcal B^*[\xi])^{r-2}+\cdots +a^l_{r-1}(\xi)\mathrm{Id}]\mathcal{B}[\xi]\\
	&\quad =- [a^l_0(\xi)(\mathcal B^*[\xi]\mathcal B[\xi])^{r}+a^l_1(\xi)(\mathcal B^*[\xi]\mathcal B[\xi])^{r-1}+\cdots +a^l_{r-1}(\xi)\mathcal B^*[\xi]\mathcal B[\xi]].
	\end{align}
	Further observe that 
	\begin{align}
	\mathcal B[\xi]\mathcal B^*[\xi]&=a^a_r(\xi)[\mathrm{Id}-\mathcal A^\dagger[\xi]\mathcal{A}[\xi]](a^a_r(\xi)[\mathrm{Id}-\mathcal A^\dagger[\xi]\mathcal{A}[\xi]])^*\\
	&= (a^a_r)^2(\xi)[\mathrm{Id}-\mathcal A^\dagger[\xi]\mathcal{A}[\xi]]
	\end{align}
	and thus 
	\begin{align}
	\mathcal G[\xi]&=a^l_r(\xi)\mathrm{Id}+\big[a^l_0(\xi)(a^a_r)^{2r}(\xi)[\mathrm{Id}-\mathcal A^\dagger[\xi]\mathcal{A}[\xi]]\\
	&+a^l_1(\xi)(a^a_r)^{2r-2}(\xi)[\mathrm{Id}-\mathcal A^\dagger[\xi]\mathcal{A}[\xi]]+\cdots +a^l_{r-1}(\xi)(a^a_r)^2(\xi)[\mathrm{Id}-\mathcal A^\dagger[\xi]\mathcal{A}[\xi]]\big]\\
	&=a^l_r(\xi)\mathrm{Id}+\big[a^l_0(\xi)(a^a_r)^{2r}(\xi)
	+a^l_1(\xi)(a^a_r)^{2r-2}(\xi)\\
	&\hspace{2cm}+\cdots +a^l_{r-1}(\xi)(a^a_r)^2(\xi)\big][\mathrm{Id}-\mathcal A^\dagger[\xi]\mathcal{A}[\xi]].\label{Gxiexpansion}
	\end{align}
	Consider the system 
	\begin{align}
	\mathcal{B}[\xi] \hat \Phi(\xi)=\hat U(\xi)\\
	\mathcal{G}[\xi] \hat \Phi(\xi)=0,
	\end{align}
	then $\mathcal{G}[\xi] \hat \Phi(\xi)=0$ implies that 
	\begin{equation}
	-a^l_r(\xi)\hat \Phi(\xi)=[a^l_0(\xi)(a^a_r)^{2r-1}(\xi)
	+a^l_1(\xi)(a^a_r)^{2r-3}(\xi)
	+\cdots +a^l_{r-1}(\xi)(a^a_r)(\xi)]\hat U(\xi)
	\end{equation}
	and thus
	\begin{equation}
	\hat \Phi(\xi)=-\frac{[a^l_0(\xi)(a^a_r)^{2r-1}(\xi)
		+a^l_1(\xi)(a^a_r)^{2r-3}(\xi)
		+\cdots +a^l_{r-1}(\xi)(a^a_r)(\xi)]}{a^l_r(\xi)}\hat U(\xi).
	\end{equation}
	From this and using \eqref{Gxiexpansion} we see that 
	\begin{equation}\label{eq:elipticformofphi}
	\hat \Phi(\xi)= -\frac{C}{a^a_r(\xi)}\hat U(\xi).
	\end{equation}
	Finally, note that $a^a_r(\xi)$ is the last term in the characteristic  polynomial of $\mathcal{A}[\xi]\mathcal{A}^*[\xi]$, which is a symmetric matrix, and $\mathcal{A}$ is homogeneous and has constant rank. Therefore, $a^a_r$ is real valued and has order $2k$ in $\xi$. Thus \eqref{eq:elipticformofphi} simplifies to 
	\begin{equation}
	\hat \Phi(\xi)= -\frac{C}{(\xi)^{2k}}\hat U(\xi)
	\end{equation}
	and so we have a relation with an elliptic symbol and can apply Calder\'on-Zygmund theory (see \cite{GrafkosCalderonZygmund} and also \cite{SteinSingularintegrals}) to the singular integral kernel generated from this symbol. Thus using techniques in \cite{MazyaSobolevspaces}, (assuming a Lipschitz domain , here a torus), we obtain 
	\begin{equation}
	\|\Phi\|_{W^{2k,p}}\le C \|U\|_{L^p}. 
	\end{equation} 
	We can use density of $C^\infty$ in $L^p$ and $W^{2k,p}$ with the bound above to see that the estimate extends to $U\in L^p$.
\end{proof}

\subsection*{Acknowledgements}

 The research funding for J.S. was graciously given by German research foundation (DFG) grant no. WI 4840/1-1.
 
 We would like to thank Bogdan Raita for valuable discussions on the topic.


\begin{thebibliography}{10}
	
	

	
	
		\bibitem{Continontinterpentration}
	S.~Conti and G.~Dolzmann.
	\newblock On the theory of relaxation in nonlinear elasticity with constraints on the determinant.
	\newblock {\em  Arch. Ration. Mech. Anal.}, {\bf 217(2)}, 413--437, 2015.

	\bibitem{CG2020}
	S.~Conti and F.~Gmeineder.
	\newblock $\mathcal A$-quasiconvexity and partial regularity.
	\newblock {\em Preprint},  arXiv:2009.13820, 2020.
		 
		 
	\bibitem{DacorognaBookgenaralquasiconvexity}
	B.~Dacorogna. 
	\newblock  {\em    Weak continuity and weak lower semi-continuity of non-linear functionals.} 
	\newblock Springer-Verlag, 1982.
	
			\bibitem{Decellgeneralisedinverse}
	H.P.~Decell, Jr. 
	\newblock  An application of the Cayley-Hamilton theorem to generalized matrix inversion. 
	\newblock {\em SIAM Rev.}, {\bf 7(4)}, 526--528, 1965.
	
			\bibitem{lazloconvexeuler}
	C.~De Lellis and L.~Sz\'ekelyhidi, Jr.
	\newblock On admissibility criteria for weak solutions of the Euler equations.
	\newblock {\em  Arch. Ration. Mech. Anal.}, {\bf 195(1)}, 225--260, 2010.
	

	
			\bibitem{RossaSerresemiconvaty}
	L.~De Rosa, 	D.~Serre and R.~Tione.
	\newblock On the upper semicontinuity of a quasiconcave functional.
	\newblock {\em  J. Funct. Anal.}, {\bf 279(7)}, Art.~108660, 2020.
	

	
		\bibitem{RossaSerreintegrability}
	L.~De Rosa and R.~Tione.
	\newblock On a question of D. Serre. 
	\newblock {\em To appear in ESAIM Control Optim. Calc. Var.}, 2020.

		\bibitem{FonescaMuelleraquasiconvexity}
	I.~Fonseca and S.~M\"uller.
	\newblock  $\mathcal A$-quasiconvexity, lower semicontinuity, and Young measures. 
	\newblock {\em  SIAM J. Math. Anal.}, {\bf 30(6)}, 1355--1390, 1999.
	
		
	\bibitem{GrafkosCalderonZygmund}
	L.~Grafakos 
	\newblock  {\em   Classical fourier analysis} (Vol. 2),   
	\newblock New York: Springer, 2008.
	
		\bibitem{Rindlernoninter}
	K.~Koumatos, F.~Rindler, and 
	E.~Wiedemann.
	\newblock Orientation-preserving Young measures.
	\newblock  {\em   Q. J. Math.},  {\bf 67(3)}, 439--466, 2016.
	
	
	
	\bibitem{MazyaSobolevspaces}
	V.~Maz'ya. 
	\newblock  {\em    Sobolev spaces},   
	\newblock Springer, 2013.
	
	
		\bibitem{Morreyquaiconvexity}
	C.~B.~Morrey. 
	\newblock Quasi-convexity and the lower semicontinuity of multiple integrals.
	\newblock {\em Pacific J. Math.}, {\bf 2(1)}, 25--53, 1952.
	
		\bibitem{Muratcompensatedcompactness}
	F.~Murat. 
	\newblock Compacit\'e par compensation.
	\newblock {\em  Ann. Scuola Norm. Sup. Pisa Cl. Sci. (4)}, {\bf 5(3)}, 489--507, 1978.
	

		\bibitem{RaitaPotentials}
	B.~Raita. 
	\newblock Potentials for A-quasiconvexity.
	\newblock {\em Calc. Var. Partial Differential Equations}, {\bf 58(3)}, 2019.
	
	\bibitem{SerreDPTs}
	D.~Serre 
	\newblock  Divergence-free positive symmetric tensors and fluid dynamics. 
	\newblock {\em Ann. Inst. H. Poincar\'e Anal. Non Lin\'eaire}, {\bf 35(5)}, 1209--1234, 2018.
	
		\bibitem{SerreMultidconservationlaws}
	D.~Serre. 
	\newblock  Multi-dimensional scalar conservation laws with unbounded integrable initial data.
	\newblock {\em arXiv:1807.10474.}, 2018.
	
	
	\bibitem{SerreClassovpoission}
	D.~Serre. 
	\newblock   Compensated integrability. Applications to the Vlasov-Poisson equation and other models in mathematical physics.
	\newblock {\em J. Math. Pures Appl.}, {\bf 127(9)}, 67--88, 2019.
	
	\bibitem{Serrecollionsmoleculardynamics}
	D.~Serre. 
	\newblock    Estimating the number and the strength of collisions in molecular dynamics.
	\newblock {\em arXiv:1903.05866.}, 2019.
	
	\bibitem{SerreSilvestermultiburgers}
	D.~Serre and L.~Silvestre.  
	\newblock    Multi-dimensional Burgers equation with unbounded initial data: well-posedness and dispersive estimates.
	\newblock {\em Arch. Ration. Mech. Anal.}, {\bf 234(3)}, 1391--1411, 2019.
	
		\bibitem{SteinSingularintegrals}
	E.~M.~Stein. 
	\newblock  {\em   Singular integrals and differentiability properties of functions} (Vol. 2).   
	\newblock Princeton University Press, 1970.
	
		\bibitem{Tartarfrenchcompensated}
	L.~Tartar. 
	\newblock   Une nouvelle m\'ethode de r\'esolution d'\'equations aux d\'eriv\'ees partielles non lin\'eaires.   
	\newblock In{ \em Journ\'ees d`Analyse Non Lin\'eaire}, 228--241, Lecture Notes in Math. {\bf 665},  Springer, Berlin, 1978.
	
			\bibitem{Tartarcompensated}
	L.~Tartar. 
	\newblock    Compensated compactness and applications to partial differential equations.
	\newblock In{ \em Nonlinear analysis and mechanics: Heriot-Watt Symposium},   {\bf 4}, 136--212, Res. Notes in Math. {\bf 39}, Pitman, Boston/London, 1979.


\end{thebibliography}
\end{document}